\documentclass[11pt,reqno]{amsart}
\usepackage{amscd,amssymb,amsmath,amsthm}
\usepackage[arrow,matrix]{xy}
\usepackage{graphicx}
\usepackage{epstopdf}
\usepackage{color}
\usepackage{multicol}
\newtheorem{thm}{Theorem}
\newtheorem{defn}{Definition}
\newtheorem{lemma}{Lemma}
\newtheorem{pro}{Proposition}
\newtheorem{rk}{Remark}

\newtheorem{ex}{Example}

\numberwithin{equation}{section} 

\begin{document}
\title [A DISCRETE-TIME DYNAMICAL SYSTEM OF MOSQUITO POPULATION]
{A DISCRETE-TIME DYNAMICAL SYSTEM OF STAGE-STRUCTURED WILD AND STERILE MOSQUITO POPULATION}

\author{ Z.S. Boxonov, U.A. Rozikov}

\address{Z.\ S.\ Boxonov\\ Institute of mathematics, 81, Mirzo Ulug`bek str., 100170,
Tashkent, Uzbekistan.}
 \email {z.b.x.k@mail.ru}

\address{U.\ A.\ Rozikov \\ Institute of mathematics, 81, Mirzo Ulug`bek str., 100170,
Tashkent, Uzbekistan.} \email {rozikovu@yandex.ru}

\begin{abstract} We study the discrete-time dynamical systems associated to a stage-structured wild and sterile mosquito population.
We describe all fixed points of the evolution operator (which depends on five parameters) of mosquito population
and show that depending on the parameters this operator may have unique, two and infinitely many fixed points.
Under some general conditions on parameters we determine type of each fixed point and
give the limit points of the dynamical system. Moreover, for a special
case of parameters we give full analysis of corresponding dynamical system.
We give some biological interpretations of our results.
\end{abstract}


 \maketitle

\section{Introduction}

In \cite{LP} the authors gave a mathematical model of mosquito dispersal,
which is a continuous-time dynamical systems of
mosquito populations. Recently, in \cite{RV} a discrete-time dynamical system, generated by an evolution
operator of this mosquito population is studied.

In this paper following \cite{J.Li} we consider another model of the mosquito population.
  Mosquitoes undergo complete metamorphosis going through four distinct stages
of development during a lifetime: egg, larva, pupa and adult \cite{Becker}. After drinking
blood, adult females lay eggs in water. Within a week, the eggs hatch into larvae
that breathe air through tubes which they poke above the surface of the
water. Larvae eat bits of floating organic matter and each other. Larvae molt four
times as they grow; after the fourth molt, they are called pupae. Pupae
also live near the surface of the water, breathing through two horn-like
tubes on their back. Pupae do not eat. When the skin splits after
a few days from a pupa, an adult emerges. The adult lives for only a few weeks and
the full life-cycle of a mosquito takes about a month \cite{Mosquito}, \cite{L.Alphey,A.Bartlet}.

Consider a wild mosquito population without the presence of
sterile mosquitoes. For the simplified stage-structured
mosquito population, we group the three aquatic stages into the
larvae class by $x$, and divide the mosquito population into
the larvae class and the adults, denoted by $y$. We assume that the density dependence exists only in the larvae stage \cite{J.Li}.

We let the birth rate, that is, the oviposition rate of adults be
$\beta(\cdot)$; the rate of emergence from larvae to adults be a
function of the larvae with the form of $\alpha(1-k(x))$,
where $\alpha>0$ is the maximum emergence rete, $0\leq
k(x)\leq 1$, with $k(0)=0, k'(x)>0$, and $\lim_{x\rightarrow
\infty}k(x)=1$, is the functional response due to
the intraspecific competition \cite{J}. We let the death rate
of larvae be a linear function, denoted by $d_{0}+d_{1}x$, and
the death rate of adults be constant, denoted by $\mu$. Then we
arrive at, in the absence of sterile mosquitoes, the following
system of equations:
 \begin{equation}\label{sys}
\left\{%
\begin{array}{ll}
    \frac{dx}{dt}=\beta(\cdot) y-\alpha(1-k(x))x-(d_{0}+d_{1}x)x,\\[3mm]
    \frac{dy}{dt}=\alpha(1-k(x))x-\mu y \\
\end{array}%
\right.\end{equation}
    We further assume a functional response for $k(x)$, as in
    \cite{J}, in the form $$k(x)=\frac{x}{1+x}.$$

    Suppose mosquito adults have no difficulty to find their mates such that no Allee effects are concerned, and hence the adults birth is constant, simply denoted as $\beta(\cdot)=\beta$ \cite{J.Li}. The interactive dynamics for the wild mosquitoes are governed by the following system:
\begin{equation}\label{system}
\left\{%
\begin{array}{ll}
    \frac{dx}{dt}=\beta y-\frac{\alpha x}{1+x}-(d_{0}+d_{1}x)x,\\[3mm]
    \frac{dy}{dt}=\frac{\alpha x}{1+x}-\mu y \\
\end{array}%
\right.\end{equation}
Denote
\begin{equation}\label{cont.thm}r_{0}=\frac{\alpha\beta}{(\alpha+d_{0})\mu}.
\end{equation}
The dynamistic generated by the system (\ref{system}) can be summarized as follows.
\begin{thm} \textbf{(Theorem 3.1 in \cite{J.Li})}: If $r_{0}\leq1$, where $r_{0}$ is defined in equation (\ref{cont.thm}), the trivial equilibrium $(0;0)$ of system (\ref{system}) is a globally asymptotically stable, and there is no positive equilibrium. If $r_{0}>1$, the trivial equilibrium $(0;0)$ is unstable, and there exists a unique positive equilibrium $(x^{(0)}, y_0)$ with
$$x^{(0)}=\frac{\sqrt{(d_{0}+d_{1})^{2}-4d_{1}(\alpha+d_{0})(1-r_{0})}-d_{0}-d_{1}}{2d_{1}}, \ \ y_{0}=\frac{\alpha x^{(0)}}{\mu(1+x^{(0)})},$$ which is a globally asymptotically stable.
\end{thm}

In this paper (as in \cite{RV}, \cite{RS}) we study the discrete time dynamical systems associated to
the system(\ref{system}).

Define the operator $W:\mathbb{R}^{2}\rightarrow \mathbb{R}^{2}$ by
\begin{equation}\label{systema}
\left\{%
\begin{array}{ll}
    x'=\beta y-\frac{\alpha x}{1+x}-(d_{0}+d_{1}x)x+x,\\[3mm]
    y'=\frac{\alpha x}{1+x}-\mu y+y
\end{array}%
\right.\end{equation}
where $\alpha >0, \beta >0, \mu >0,\ d_{0}\geq0,\ d_{1}\geq0.$

We would like to study dynamical systems
corresponding to the operator (\ref{systema}).

The paper is organized as follows. In Section 2 we describe all fixed points of the  operator (\ref{systema}) of mosquito population
and show that depending on the parameters this operator may have unique, two and  infinitely many fixed points
(laying on the graph of a continuous function). In Section 3 we determine type of each fixed point and
give the limit points of the dynamical system under some general conditions on parameters. In Section 4 we consider a special
case of parameters and give full analysis of corresponding dynamical system.
In the last section we give some biological interpretations of the results.

\section{Fixed points}

Let $\mathbb{R}_{+}^{2}=\{(x,y): x,y\in\mathbb{R}, x\geq0, y\geq0\}.$
A point $z\in\mathbb{R}_{+}^{2}$ is called a fixed point of $W$ if
$W(z)=z$. The set of fixed points is denoted by Fix$(W)$.

Let us find fixed points of the operator $W.$ For this we solve the following system
\begin{equation}\label{Fsystema}
\left\{%
\begin{array}{ll}
    x=\beta y-\frac{\alpha x}{1+x}-(d_{0}+d_{1}x)x+x,\\[3mm]
    y=\frac{\alpha x}{1+x}-\mu y+y \\
\end{array}%
\right.\end{equation}
i.e.,
\begin{equation}\label{fixed}
\left\{%
\begin{array}{ll}
    \beta y=\frac{\alpha x}{1+x}+(d_{0}+d_{1}x)x,\\[3mm]
    \mu y=\frac{\alpha x}{1+x}\\
\end{array}%
\right.\end{equation}
Independently from parameters this system has a solution $(0,0)$.
To find other solutions, from the second equation of (\ref{fixed}) we obtain $y=\frac{\alpha x}{\mu (1+x)}$.

Denote $\gamma(x)=\frac{\alpha x}{\mu (1+x)}.$

Then from the first equation we get the following form
\begin{equation}\label{fix*}
d_{1}x^{2}+(d_{0}+d_{1})x+d_{0}+\alpha(1-\frac{\beta}{\mu})=0
\end{equation}
There are the following cases:

a) Let $d_{1}=0.$ Then (\ref{fix*}) has the form
\begin{equation}\label{fix2}
d_{0}x+d_{0}+\alpha(1-\frac{\beta}{\mu})=0
\end{equation}

a.1) if $d_{0}=0$ and $\beta=\mu$ then (\ref{fix2}) has infinitely many roots.

a.2) if $d_{0}=0$ and $\beta\neq\mu$ then (\ref{fix2}) has no roots.

a.3) if $d_{0}\neq0$ and $\beta =\mu(1+\frac{d_{0}}{\alpha})$ then $x=0.$

a.4) if $d_{0}\neq0$ and $\beta > \mu(1+\frac{d_{0}}{\alpha})$ then $x=\frac{\alpha(\beta-\mu)}{\mu d_{0}}-1\in\mathbb{R}_{+}\setminus \{0\}.$

b) Let $d_{1}\neq0$. The discriminant of (\ref{fix*}) is
$$\Delta=(d_{0}-d_{1})^{2}+\frac{4\alpha d_{1}(\beta -\mu)}{\mu}\geq0.$$

b.1) if $\Delta <0$ then (\ref{fix*}) has no roots.

b.2) if $\Delta=0$ then $x=-\frac{d_{0}+d_{1}}{2 d_{1}}\not\in\mathbb{R}_{+}.$

b.3) if $\Delta>0, \beta > \mu(1+\frac{d_{0}}{\alpha})$ then  $x=\frac{\sqrt{\Delta}-d_{0}-d_{1}}{2 d_{1}}\in\mathbb{R}_{+}\setminus \{0\}$.

Denote $$\Omega=\{(\alpha,\beta,\mu,d_{0},d_{1})\in\mathbb{R}^{5}:\alpha>0,\beta>0,\mu>0, d_{0}\geq0, d_{1}\geq0\},$$
$$\Phi_{1}=\{(\alpha,\beta,\mu,d_{0},d_{1})\in\Omega: d_{0}\neq0, d_{1}=0, \beta>\mu(1+\frac{d_{0}}{\alpha})\},$$
$$\Phi_{2}=\{(\alpha,\beta,\mu,d_{0},d_{1})\in\Omega: d_{1}\neq0, \beta>\mu(1+\frac{d_{0}}{\alpha})\},$$
$$\Psi=\{(\alpha,\beta,\mu,d_{0},d_{1})\in\Omega: d_{0}=d_{1}=0,\beta=\mu\},$$
$$\Omega^{*}=\Omega\setminus(\Phi_{1}\cup\Phi_{2}\cup\Psi).$$
Summarizing we formulate the following
\begin{thm}\label{fixthm}
\begin{itemize}
\item[a.] \textbf{Uniqueness of fixed point:}\ If $(\alpha,\beta,\mu,d_{0},d_{1})\in\Omega^{*}$ then the operator (\ref{systema}) has a unique fixed point $(0,0).$
\item[b.] \textbf{Two fixed points}, $(x_{i},\gamma(x_{i}))$, with $\gamma(x)=\frac{\alpha x}{\mu (1+x)}, i=1,2:$
\item[b.1)] If $(\alpha,\beta,\mu,d_{0},d_{1})\in\Phi_{1}$ then mapping (\ref{systema}) has two
fixed points with $$x_{1}=0, \ \ x_{2}=\frac{\alpha(\beta-\mu)}{\mu
d_{0}}-1.$$
\item[b.2)] If $(\alpha,\beta,\mu,d_{0},d_{1})\in\Phi_{2}$ then
the fixed points are with
 $$x_{1}=0, \ \ x_{2}=\frac{\sqrt{\Delta}-d_{0}-d_{1}}{2 d_{1}}.$$
\item[\textbf{c}.] If $(\alpha,\beta,\mu,d_{0},d_{1})\in\Psi$ then any point $(x,\gamma(x)), x\in \mathbb{R}_{+}$ is a fixed points of (\ref{systema}).
\end{itemize}
\end{thm}

\section{Types of the fixed points}

To interpret values of $x$ and $y$ as probabilities we assume
$x\geq0$ and $y\geq0$. Moreover, to define a dynamical system we need that $W$
maps  $\mathbb{R}_{+}^{2}$ to itself. It is easy to see that if
\begin{equation}\label{parametr}
\alpha\leq1-d_0, \ \ \beta >0, \ \ 0<\mu\leq1,\  \ 0\leq d_{0}<1,\ \ d_{1}=0
\end{equation}
then operator (\ref{systema}) maps $\mathbb{R}_{+}^{2}$ to itself. In this case the system (\ref{systema}) becomes
\begin{equation}\label{syst}
\left\{%
\begin{array}{ll}
    x'=\beta y-(\frac{\alpha }{1+x}+d_{0}-1)x,\\[3mm]
    y'=\frac{\alpha x}{1+x}+(1-\mu)y
\end{array}%
\right.\end{equation}
Now we shall examine the type of the fixed points.
\begin{defn}\label{d1}
(see\cite{D}) A fixed point $s$ of the operator $W$ is called
hyperbolic if its Jacobian $J$ at $s$ has no eigenvalues on the
unit circle.
\end{defn}

\begin{defn}\label{d2}
(see\cite{D}) A hyperbolic fixed point $s$ called:

1) attracting if all the eigenvalues of the Jacobi matrix $J(s)$
are less than 1 in absolute value;

2) repelling if all the eigenvalues of the Jacobi matrix $J(s)$
are greater than 1 in absolute value;

3) a saddle otherwise.
\end{defn}
To find the type of a fixed point of the operator (\ref{syst})
we write the Jacobi matrix:

$$J(z)=J_{W}=\left(%
\begin{array}{cc}
  1-d_{0}-\frac{\alpha}{(1+x)^2} & \beta \\
  \frac{\alpha}{(1+x)^2} & 1-\mu \\
\end{array}%
\right).$$

The eigenvalues of the Jacobi matrix are
$$\lambda_{1,2}=\frac{1}{2}\left(2-g(x)\pm \sqrt{f(x)}\right),$$
where $g(x)=\mu+d_{0}+\frac{\alpha}{(1+x)^2},$
$f(x)=(\mu-d_{0}-\frac{\alpha}{(1+x)^2})^{2}+\frac{4\alpha
\beta}{(1+x)^2}$.

If
\begin{equation}\label{attrac1}
|\lambda_{1,2}|=|\frac{1}{2}\left(2-g(x)\pm \sqrt{f(x)}\right)|<1
\end{equation}
then fixed points are attractive.

The inequality (\ref{attrac1}) is equivalent to the following
\begin{equation}\label{attrac2}
\left\{%
\begin{array}{ll}
    0<g(x)\leq2  \\[2mm]
    \sqrt{f(x)}<g(x)
\end{array}%
\right. or \ \left\{%
\begin{array}{ll}
    2<g(x)<4 \\[2mm]
    \sqrt{f(x)}<4-g(x)
\end{array}%
\right.
\end{equation}
The fixed points are repelling  if
\begin{equation}\label{reppel1}
|\lambda_{1,2}|=|\frac{1}{2}\left(2-g(x)\pm \sqrt{f(x)}\right)|>1.
\end{equation}
The inequality (\ref{reppel1}) is equivalent to the following
\begin{equation}\label{reppel2}
\left\{%
\begin{array}{ll}
    g(x)<0  \\[2mm]
    \sqrt{f(x)}<-g(x)
\end{array}%
\right. or \ \left\{%
\begin{array}{ll}
    g(x)>4 \\[2mm]
    \sqrt{f(x)}<g(x)-4
\end{array}%
\right.
\end{equation}
Denote
$$\Theta=\{(\alpha,\beta,\mu,d_{0},d_{1})\in\Omega:d_{1}=0, \alpha\leq1-d_{0},0<\mu\leq1, 0\leq d_{0}<1\},$$
$$\Theta_{1}=\{(\alpha,\beta,\mu,d_{0},d_{1})\in\Omega:d_{1}=0, \mu+d_{0}+\alpha\leq2, \beta<\mu(1+\frac{d_{0}}{\alpha})\},$$
$$\Theta_{2}=\{(\alpha,\beta,\mu,d_{0},d_{1})\in\Omega:d_{1}=0, \mu+d_{0}+\alpha\leq2,\beta>\mu(1+\frac{d_{0}}{\alpha})\},$$
Note that the set $\Theta$ is the condition (\ref{parametr}), to work under this condition we need to introduce the following sets
$$\Theta^{*}=\Omega^{*}\cap\Theta, \ \ \Phi^{*}=\Theta\cap\Phi_{1}, \ \ \Psi^{*}=\Theta\cap\Psi.$$

By solving (\ref{attrac2}) and (\ref{reppel2}) at each fixed point we obtain the following
\begin{thm} \begin{itemize}
\item[\emph{a}.] The type of the unique fixed point, $(0,0)$, for (\ref{syst}) is as follows:
$$(0,0)\  is \left\{\begin{array}{ll}
    attractive \ \ \hbox{if}\ (\alpha,\beta,\mu,d_{0},d_{1})\in\Theta^{*}\cap\Theta_{1}, \\[2mm]
    saddle \ \ \hbox{if}\ \ (\alpha,\beta,\mu,d_{0},d_{1})\in\Theta^{*}\setminus\Theta_{1}.
\end{array}\right.$$
\item[\emph{b}.] The point $(x_{1},y_{1})=(0,0)$ is saddle if $(\alpha,\beta,\mu,d_{0},d_{1})\in\Phi^{*}$ and
$$(x_{2},\gamma(x_{2}))\ with\ x_{2}=\left\{\begin{array}{ll}
    \frac{\alpha(\beta-\mu)}{\mu d_{0}}-1 \ \hbox{is attractive} \ \ \hbox{if}\ (\alpha,\beta,\mu,d_{0},d_{1},d_{1})\in\Phi^{*}\cap\Theta_{2}, \\[2mm]
    \frac{\alpha(\beta-\mu)}{\mu d_{0}}-1 \ \hbox{is saddle} \ \ \hbox{if}\ (\alpha,\beta,\mu,d_{0},d_{1},d_{1})\in\Phi^{*}\setminus\Theta_{2}.
    \end{array}\right.$$
\item[\emph{c}.] For any $x\in\mathbb{R_{+}}$ the point $(x,\gamma(x))$ is saddle if $(\alpha,\beta,\mu,d_{0},d_{1})\in \Psi^{*}.$
\end{itemize}
\end{thm}
\begin{rk} We have the following \begin{itemize}
\item if $(\alpha,\beta,\mu,d_{0},d_{1})\in\Theta$ then the fixed point with $x_{2}=\frac{\sqrt{\Delta}-d_{0}-d_{1}}{2 d_{1}}$ is outside of $\mathbb R_+^2$

\item if $(\alpha,\beta,\mu,d_{0},d_{1})\in \Theta^*$ then Fix$(W)\cap\{(x,\gamma(x)):x\in\mathbb{R_{+}}\}=\{(0,0)\}.$
\end{itemize}
\end{rk}
From the known theorems (see \cite{D} and  \cite{G})
 we get the following result
 \begin{pro}\label{pr} For the operator $W$ given by (\ref{systema}), under condition (\ref{parametr}) the following holds
 $$\lim_{n\to \infty}W^n(z_0)=\left\{\begin{array}{ll}
 (0,0), \ \ \mbox{if} \ \ \beta\leq \mu\left(1+{d_0\over \alpha}\right), \ \ \mbox{and} \ \ z_0\in \mathbb U_0\\[2mm]
 (x^*,y^*), \ \ \mbox{if} \ \ \beta> \mu\left(1+{d_0\over \alpha}\right), \ \ \mbox{and} \ \ z_0\in \mathbb U^*
 \end{array}\right.$$
 where $W^n$ is $n$-th iteration of $W$, $x^*={\alpha(\beta-\mu)\over \mu d_0}-1$, $y^*={\alpha x^*\over \mu(1+x^*)}$ and $\mathbb U_0$ is a neighborhood of $(0,0)$,
 $\mathbb U^*$ is a neighborhood of $(x^*,y^*)$.
  \end{pro}

In the following examples we show that if the condition (\ref{parametr}) is not satisfied then
the dynamical system may have several kind of limit points.

\begin{ex} With parameters $(\alpha, \beta, \mu, d_{0}, d_{1}) =(1.5, 0.4, 0.5, 0, 0)$ belong to the set $\Omega^{*}\setminus\Theta.$
If the initial point is $(x_{0}, y_{0})=(5,4)$ then the trajectory of system (\ref{systema}) is shown in the Fig. \ref{fig.1}, i.e., it converges to $(0,0)$.
\end{ex}
\begin{ex} With parameters $(\alpha, \beta, \mu, d_{0}, d_{1}) =(1.5, 0.5, 0.4, 0, 0)$ belong to set $\Omega^{*}\setminus\Theta.$ If the initial point is $(x_{0}, y_{0})=(10,9)$ then the trajectory of system (\ref{systema}) is shown in the Fig. \ref{fig.2}. In this case the first coordinate of the trajectory
goes to infinite and the second coordinate has limit point approximately $3.75$.
\end{ex}
\begin{ex} With parameters $(\alpha, \beta, \mu, d_{0}, d_{1}) =(6, 0.5, 0.4, 0.6, 0)$ belong to set $\Phi_{1}\setminus\Theta.$ If the initial point is $(x_{0}, y_{0})=(50,80)$ then the trajectory of system (\ref{systema}) is shown in the Fig.\ref{fig.3}, i.e., it converges to the fixed point $(1.5, 9)$.
\end{ex}

\begin{figure}[h!]
\begin{multicols}{2}
\hfill
\includegraphics[width=0.45\textwidth]{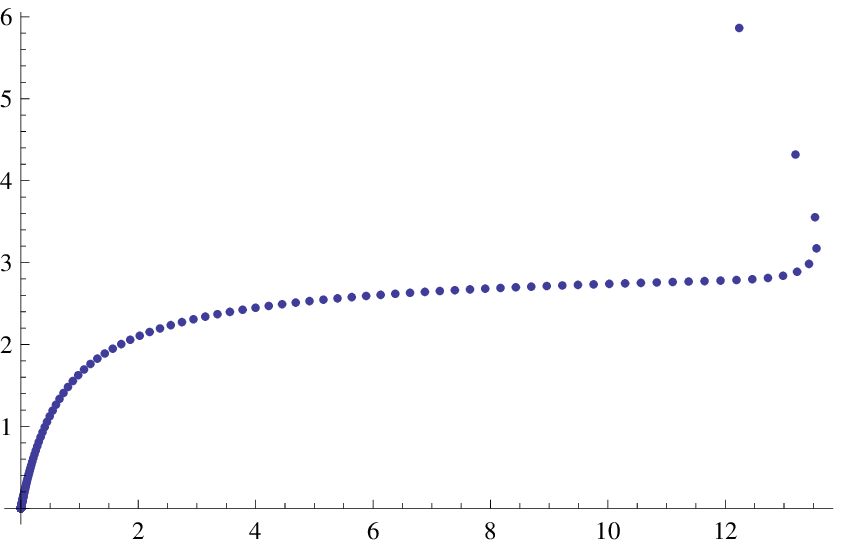}\\
\hfill
\caption{$(\alpha, \beta, \mu, d_{0}, d_{1}) =(1.5, 0.4, 0.5, 0, 0)$}\label{fig.1}
\includegraphics[width=0.45\textwidth]{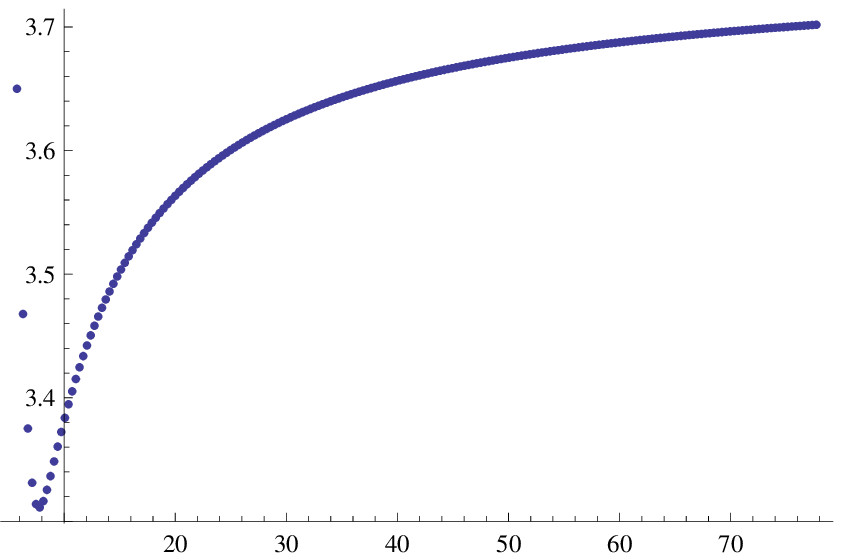}\\
\hfill
\caption{$(\alpha, \beta, \mu, d_{0}, d_{1}) =(1.5, 0.5, 0.4, 0, 0)$}\label{fig.2}

\end{multicols}
\end{figure}

\begin{figure}[h!]
  \includegraphics[width=0.5\textwidth]{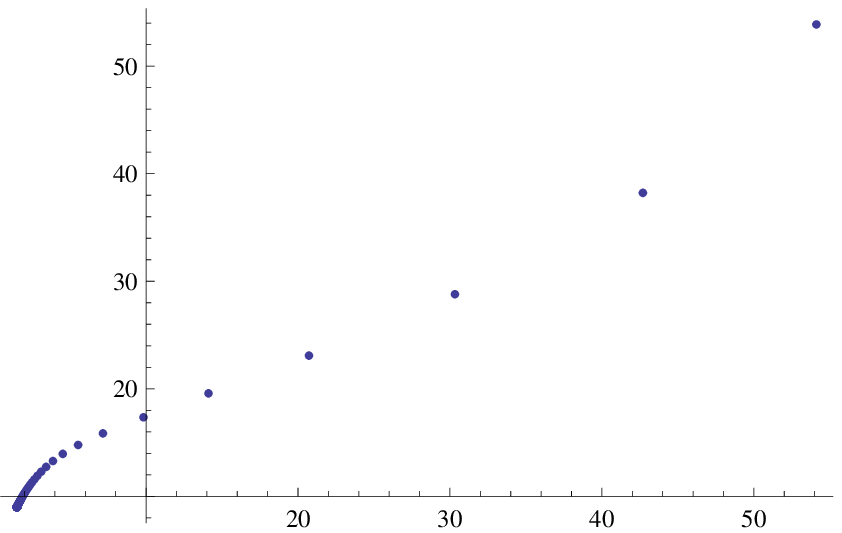}\\
  \caption{$(\alpha, \beta, \mu, d_{0}, d_{1}) =(6, 0.5, 0.4, 0.6, 0)$}\label{fig.3}
\end{figure}

\section{Dynamics for a special case}
In this section we assume
$$\beta=\mu, \ \ d_{0}=d_{1}=0$$
 then (\ref{systema}) has the following form
\begin{equation}\label{systemacase1}
W_{0}:\left\{%
\begin{array}{ll}
    x'=\beta y-\frac{\alpha x}{1+x}+x,\\[2mm]
    y'=\frac{\alpha x}{1+x}-\beta y+y.
\end{array}%
\right.\end{equation}

\begin{figure}[h!]
  \includegraphics[width=0.5\textwidth]{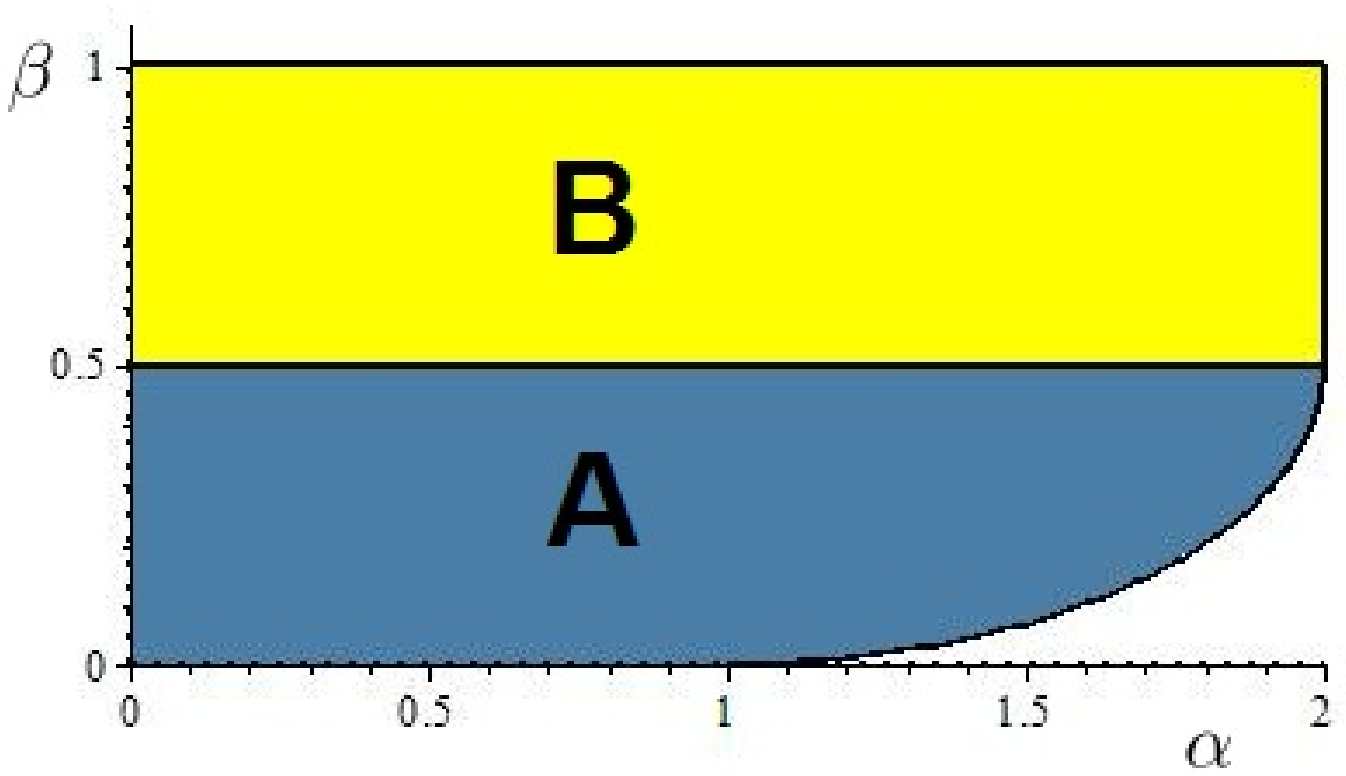}\\
  \caption{ }\label{}
\end{figure}
We denote
$$S=\{(x,y) \, |\, x,y\in\mathbb{R_{+}}, x+y=1\},$$
$$A=\{(\alpha;\beta):\ \beta\in(0;\frac{1}{2}),\ \alpha\in(0;1+2\sqrt{\beta(1-\beta)}]\},$$
$$B=\{(\alpha;\beta):\ \beta\in[\frac{1}{2};1],\ \alpha\in(0;2]\}.$$
The following lemma is useful
\begin{lemma} The operator $W_{0}$ maps the set $S$ to itself if and only if $(\alpha,\beta)\in A\cup B$.
\end{lemma}
\begin{proof} Necessity. Let $z=(x,y)\in S.\ z'=W_{0}(z)=(x', y')$.
If we add equations of (\ref{systemacase1}) then $x'+y'=x+y=1$. So, we have $y=1-x$ and $x'=\beta(1-x)-\frac{\alpha x}{1+x}+x.$
If $x=0$ then $x'=\beta$. Since $x'\in[0;1]$ we get $\beta \in(0;1]$.
If $x=1$ then $x'=1-\frac{\alpha}{2}$ and $\alpha \in(0;2]$ because $x'\in[0;1]$.
Moreover, for $x'\in [0;1]$ it should be true the inequalities $0\leq \beta(1-x)-\frac{\alpha x}{1+x}+x\leq1$. Let us write these inequalities as
\begin{equation}\label{qo`sh}
\left\{
  \begin{array}{ll}
    (1-\beta)x^{2}+(1-\alpha)x+\beta \geq0,\\[2mm]
    (1-\beta)x^{2}-\alpha x+\beta-1\leq0.  \end{array}
\right.\end{equation}
The second inequality in (\ref{qo`sh}) is always true for all $\alpha>0$ and $\beta\in(0;1]$. Hence we solve
the first inequality under conditions $\alpha>0$ and $\beta\in(0; 1]$.
\begin{itemize}
  \item[\textbf{i)}] if  $\beta\in(0;1],\ (1-\alpha)^{2}-4\beta(1-\beta)\leq0$ for (\ref{qo`sh}), then $x\in[0;1]$. Also we obtain $\alpha\in[1-2\sqrt{\beta(1-\beta)};1+2\sqrt{\beta(1-\beta)}]$.
  \item[\textbf{ii)}] if $\beta\in(0;1],\ (1-\alpha)^{2}-4\beta(1-\beta)\geq0$ for (\ref{qo`sh}), then $x\in[0;1]\subset(-\infty;x^{(1)}]$ or $x\in[0;1]\subset[x^{(2)};\infty)$.
 So, it is possible that $x^{(1)}=\frac{\alpha-1-\sqrt{(1-\alpha)^{2}-4\beta(1-\beta)}}{2(1-\beta)}\geq1$ or $x^{(2)}=\frac{\alpha-1+\sqrt{(1-\alpha)^{2}-4\beta(1-\beta)}}{2(1-\beta)}\leq0$.

From the inequality $x^{(1)}\geq1$ we get
$$\left\{
  \begin{array}{ll}
    \alpha-1-2(1-\beta)\geq \sqrt{(1-\alpha)^{2}-4\beta(1-\beta)},\\[2mm]
   (1-\alpha)^{2}-4\beta(1-\beta)\geq0.
  \end{array}
\right.$$
Consequently, $\beta\in[\frac{1}{2};1],\ \alpha\in[1+2\sqrt{\beta(1-\beta)};2]$.

From the inequality $x^{(2)}\leq0$
$$\left\{
  \begin{array}{ll}
    \sqrt{(1-\alpha)^{2}-4\beta(1-\beta)}\leq 1-\alpha, \\[2mm]
    (1-\alpha)^{2}-4\beta(1-\beta)\geq0
  \end{array}
\right.$$
and it follows that $\beta\in(0;1],\ \alpha\in(0;1-2\sqrt{\beta(1-\beta)}]$.
\end{itemize}

So, $x\in[0;1]$ holds for (\ref{qo`sh}) when $\alpha>0$ and $\beta\in(0;1]$ and in cases \textbf{i), ii)}  it holds $\beta\in(0;1], \alpha\in[1-2\sqrt{\beta(1-\beta)};1+2\sqrt{\beta(1-\beta)}]$ or $\beta\in[\frac{1}{2};1],\ \alpha\in[1+2\sqrt{\beta(1-\beta)};2]$ or $\beta\in(0;1],\ \alpha\in(0;1-2\sqrt{\beta(1-\beta)}]$ for parameters $\alpha$ and $\beta$.

Sufficiency. It is easy to check if $(\alpha,\beta)\in A\cup B$ then $W_{0}:S\to S$.
\end{proof}

The restriction on $S$ of the operator $W_{0}$, denoted by $U$, has the form
\begin{equation}\label{case1}
U:x'=\beta(1-x) -\frac{\alpha x}{1+x}+x.
\end{equation}
$U:S_{1}\rightarrow S_{1}$, $S_{1}=\{x:x\in[0;1]\}.$

For fixed point of $U$ the following lemma holds.

\begin{lemma}\label{lemmafix}
(\ref{case1}) has unique fixed point
$x^{*}=\frac{\sqrt{\alpha^{2}+4\beta^{2}}-\alpha}{2\beta}$.
\end{lemma}
\begin{proof} We need to solve $x=\beta(1-x) -\frac{\alpha x}{1+x}+x$. It is easy to see that $x_{1,2}=\frac{-\alpha\pm
\sqrt{\alpha^{2}+4\beta^{2}}}{2\beta}$ are roots. Since these roots should be in $S_{1}$, one checks that
$x_{1}=\frac{-\alpha+\sqrt{\alpha^{2}+4\beta^{2}}}{2\beta}\in S_{1},\
 \ x_{2}=\frac{-\alpha-\sqrt{\alpha^{2}+4\beta^{2}}}{2\beta}\not\in S_{1}.$
\end{proof}

Suppose $x_{0}$ is a fixed point for $U$. For one dimensional dynamical systems it is known that $x_{0}$ is an
attracting fixed point if $|U'(x_{0})|<1$. The point $x_{0}$ is
a repelling fixed point if $|U'(x_{0})|>1$. Finally, if
$|U'(x_{0})|=1$, the fixed point is saddle \cite{D}.

Let us calculate the derivative of $U(x)$ at the fixed point $x^{*}$.

Consider
$$1+x^{*}=1+\frac{-\alpha+\sqrt{\alpha^{2}+4\beta^{2}}}{2\beta}=\frac{2\alpha}{\alpha-2\beta+\sqrt{\alpha^{2}+4\beta^{2}}}.$$
Thus we obtain
$U'(x^{*})=1-\beta-\frac{\alpha}{(1+x^{*})^{2}}=1-\frac{\alpha^{2}+4\beta^{2}+(\alpha-2\beta)\sqrt{\alpha^{2}+4\beta^{2}}}{2\alpha}$.
\begin{itemize}
  \item[1)] Let $|U'(x^{*})|<1$. Then we have $\beta(2\beta-\alpha)+(1-\beta)(2\beta-\alpha+\sqrt{\alpha^{2}+4\beta^{2}})>0$.
  This inequality is always true in $(\alpha;\beta)\in (A\cup B)\setminus\{(2;1)\}$.
  \item [2)]Let $|U'(x^{*})|=1$.  If $U'(x^{*})=1$, then $\alpha^{2}+4\beta^{2}+(\alpha-2\beta)\sqrt{\alpha^{2}+4\beta^{2}}=0$. So we have $\alpha=0,\ \beta=0$. If $U'(x^{*})=-1$, then $\alpha^{2}-4\alpha+4\beta^{2}+(\alpha-2\beta)\sqrt{\alpha^{2}+4\beta^{2}}=0$. This equality holds when $\alpha=2,\ \beta=1$.
  \item [3)] Let $|U'(x^{*})|>1$. This inequality does not hold in $(\alpha;\beta)\in (A\cup B)\setminus\{(2;1)\}$.
\end{itemize}
For type of $x^{*}$ the following lemma holds.
\begin{lemma}\label{atr} The type of the fixed point $x^{*}$  for (\ref{case1}) are as follows:
\begin{itemize}
  \item[i)] if\ $(\alpha;\beta)\in (A\cup B)\setminus\{(2;1)\}$ then $x^{*}$ is attracting;
  \item[ii)]if\ $\alpha=2$ and $\beta=1$ then $x^{*}$ is saddle;
  \end{itemize}
\end{lemma}
\textbf{Periodic points}

A point $z$ in $W_{0}$ is called periodic point of $W_{0}$ if there exists $p$ so that $W_{0}^{p}(z)=z$. The smallest positive integer $p$ satisfy $W_{0}^{p}(z)=z$ is called the prime period or least period of the point $z.$
Denote by $Per_{p}(W_{0})$
the set of periodic points with prime period $p.$

Let us first describe periodic points with $p=2$ on $S,$ in this case the equation $W_{0}(W_{0}(z))=z$
can be reduced to description of 2-periodic points of the function $U$ defined in (\ref{case1}), i.e.,to solution of the equation
\begin{equation}\label{per.2}
U(U(x))=x.\end{equation}
Note that the fixed points of $U$ are solutions to (\ref{per.2}), to find other solution we consider the equation
$$\frac{U(U(x))-x}{U(x)-x}=0,$$
simple calculations show that the last equation is equivalent to the following
\begin{equation}\label{per.22}(1-\beta)x^{2}+(2-\alpha)x+1+\beta+\frac{\alpha}{\beta-2}=0.\end{equation}
Solutions to (\ref{per.22}) are $x_{1,2}=\frac{\alpha-2\pm\sqrt{D}}{2(1-\beta)}$.
Since $\alpha\in(0;2]$ we have $x_{1}=\frac{\alpha-2-\sqrt{D}}{2(1-\beta)}<0$ and $x_{1}\not\in S_{1}.$ It holds that $x_{2}=\frac{\alpha-2+\sqrt{D}}{2(1-\beta)}\in S_{1}$  when  parameters $\alpha$ and $\beta$ satisfy the attitude $(1+\beta)(2-\beta)\leq \alpha \leq \frac{4(2-\beta)}{3-\beta}$. This attitude holds when parameters $(\alpha;\beta)\in A\cup B$ are only $\alpha=2,\ \beta=1$.

Thus we have
\begin{lemma}\label{perlemma} The set of two periodic points of (\ref{case1}):
\begin{itemize}
  \item if $\alpha=2,\ \beta=1$ then $Per_{2}(U)=S_{1}$
  \item if $(\alpha;\beta)\in (A\cup B)\setminus\{(2;1)\}$ then $Per_{2}(U)=\emptyset$
\end{itemize}
\end{lemma}

The following  describes the trajectory of any point $x_{0}$ in $S_{1}$.
\begin{lemma}\label{orbita}
Let $x_{0}\in S_{1}$ be an initial point
\begin{itemize}
  \item[1)] If \  $(\alpha;\beta)\in (A\cup B)\setminus\{(2;1)\}$  then $ \lim_{m\to \infty}U^{m}(x_{0})=x^{*}.$
  \item[2)] If\ $\alpha=2,\beta=1$ then
  $$\lim_{n\to \infty}U^{n}(x_{0})=\left\{\begin{array}{ll}
  x_{0},\ \ \mbox{for} \ \ n=2k, k=0,1,2,...\\[2mm]
\frac{1-x_{0}}{1+x_{0}},\ \ \ \mbox{for} \ \ n=2k-1
\end{array}\right.$$
\end{itemize}
\end{lemma}
\begin{proof}
Denote $$ C=\{(\alpha;\beta):\ \beta\in(0;1],\ \alpha\in(0;1-\beta]\},$$
$$D=\{(\alpha;\beta):\ \beta\in[\frac{1}{2};1],\ \alpha\in[4(1-\beta);2]\},$$
$$E=\{(\alpha;\beta):\ \beta\in(0;1],\ \alpha\in(1-\beta;2(1-\beta)]\},$$
$$F=\{(\alpha;\beta):\ \beta\in(0;1],\ \alpha\in[2(1-\beta);4(1-\beta)]\cap[0;2]\},$$
$$E^{*}=(A\cup B)\cap E,\ F^{*}=(A\cup B)\cap F,$$
where $A\cup B=C\cup D\cup E^{*}\cup F^{*},\ C\cap E^{*}=\emptyset,\ C\cap F^{*}=\emptyset,\ C\cap D=\emptyset,\ E^{*}\cap F^{*}=\emptyset,\ E^{*}\cap D=\emptyset,\ D\cap F^{*}=\emptyset$

If $(\alpha;\beta)\in A\cup B$, then $U(x)$ maps $S_{1}$ to itself.

Let us find minimum points of $U(x)$. By solving $U'(x)=0$ we have
$x_{min}=\sqrt{\frac{\alpha}{1-\beta}}-1$.
If  $(\alpha;\beta)\in C\cup D$ then $x_{min}\not\in int S_{1}=\{(x,y)\in (0;1)^2: x+y=1\}$ if $(\alpha;\beta)\in E^{*}\cup F^{*}$ then $x_{min}\in int S_{1}.$
\begin{itemize}
\item[1)] \textbf{Case:} $x_{min}\not\in int S_{1}.$ (cf. with proof of Lemma 3.4 of \cite{RAU}) In this case for set $C$ we have $U'(x)>0$, i.e.,\ $U$ is an increasing function (see Fig. \ref{rC}). Here we consider the case when the function $U$ has unique fixed point $x^{*}.$ We have that the point $x^{*}$ is attractive, i.e., $|U'(x^*)|<1.$ Now we shall take arbitrary $x_{0}\in S_{1}$ and prove that $x_{n}=U(x_{n-1}),\ n\geq1$ converges as $n\rightarrow\infty.$ Consider the following partition $[0;1]=[0;x^{*})\cup\{x^{*}\}\cup(x^{*};1].$ For any $x\in[0;x^{*})$ we have $x<U(x)<x^{*},$ since $U$ is an increasing function, from the last inequalities we get $x<U(x)<U^{2}(x)<U(x^{*})=x^{*}$ iterating this argument we obtain $U^{n-1}(1)<U^{n}(x)<x^{*},$ which for any $x_{0}\in[0;x^{*})$ gives $x_{n-1}<x_{n}<x^{*},$ i.e.,$x_{n}$ converges and its limit is a fixed point of $U,$ since $U$ has unique fixed point $x^{*}$ in $[0;x^{*}]$ we conclude that the limit is $x^{*}.$ For $x\in(x^{*};1]$ we have $1>x>U(x)>x^{*},$ consequently $x_{n}>x_{n+1},$ i.e.,\ $x_{n}$ converges and its limit is again $x^{*}.$

Note that $D\subset B.$ For set $D$ we have $U'(x)<0$, i.e.,\ $U$ is a decreasing function (see Fig. \ref{rD}). Let $g(x)=U(U(x)).$ $g$ is increasing since $g'(x)=U'(U(x))U'(x)>0.$ By Lemma \ref{lemmafix} and Lemma \ref{perlemma} we have that $g$ has at most unique fixed point (including $x^{*}$). Hence one can repeat the same argument of the proof of part $1)$ for the increasing function $g$ and complete the proof.

\textbf{Case:} $x_{min}\in int S_{1}.$

a) Let $x_{min}<x^{*}.$ Consider the following partition $[0;1]=[0;x_{min})\cup[x_{min};1].$
  The function $U(x)$ is decreasing in $[0;x_{min})$ and is increasing in $[x_{min};1]$ (see Fig. \ref{rE}). For all $x\in[0;x_{min})$, $x<U(x),\ U(x)>x_{min}.$  For $U(x)\in [x_{min};1]$ it can be proved that $x_{n}$ converges to the attractive fixed point $x^{*}$ (see Lemma \ref{atr}) like previous case.

b) Let $x_{min}>x^{*}.$ Consider the following partition $[0;1]=[0;x_{min})\cup[x_{min};1].$
For all $x\in[x_{min};1]$, $x>U(x)>U^{2}(x)>...>U^{k}(x),\ U^{k}(x)<x_{min}$ (see Fig. \ref{rF}).
 If $U(x)\in [0;x_{min})$ then the sequence $x_{n}$ converges to $x^{*}$.

\item[2)]
When $\alpha=2,\ \beta=1$  the function becomes $U(x)=\frac{1-x}{1+x}$. Besides,  $U^{2n}(x)=x,\ U^{2n-1}(x)=\frac{1-x}{1+x}.$
This completes the proof.
\end{itemize}
\end{proof}

\begin{figure}[h!]
\begin{multicols}{2}
\hfill
\includegraphics[width=0.4\textwidth]{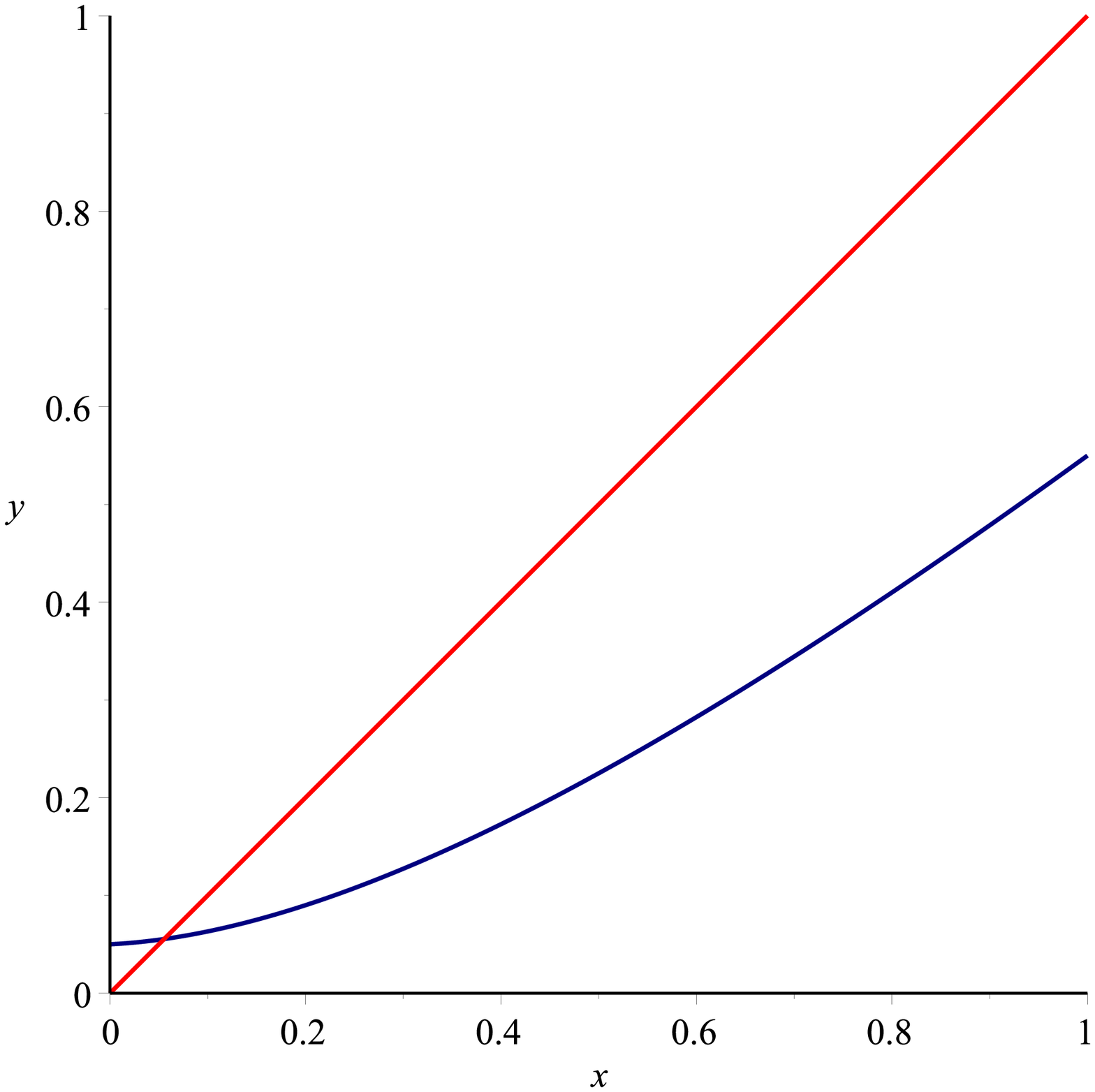}\\
\hfill
\caption{$(\alpha;\beta)\in C$}\label{rC}
\hfill
\includegraphics[width=0.4\textwidth]{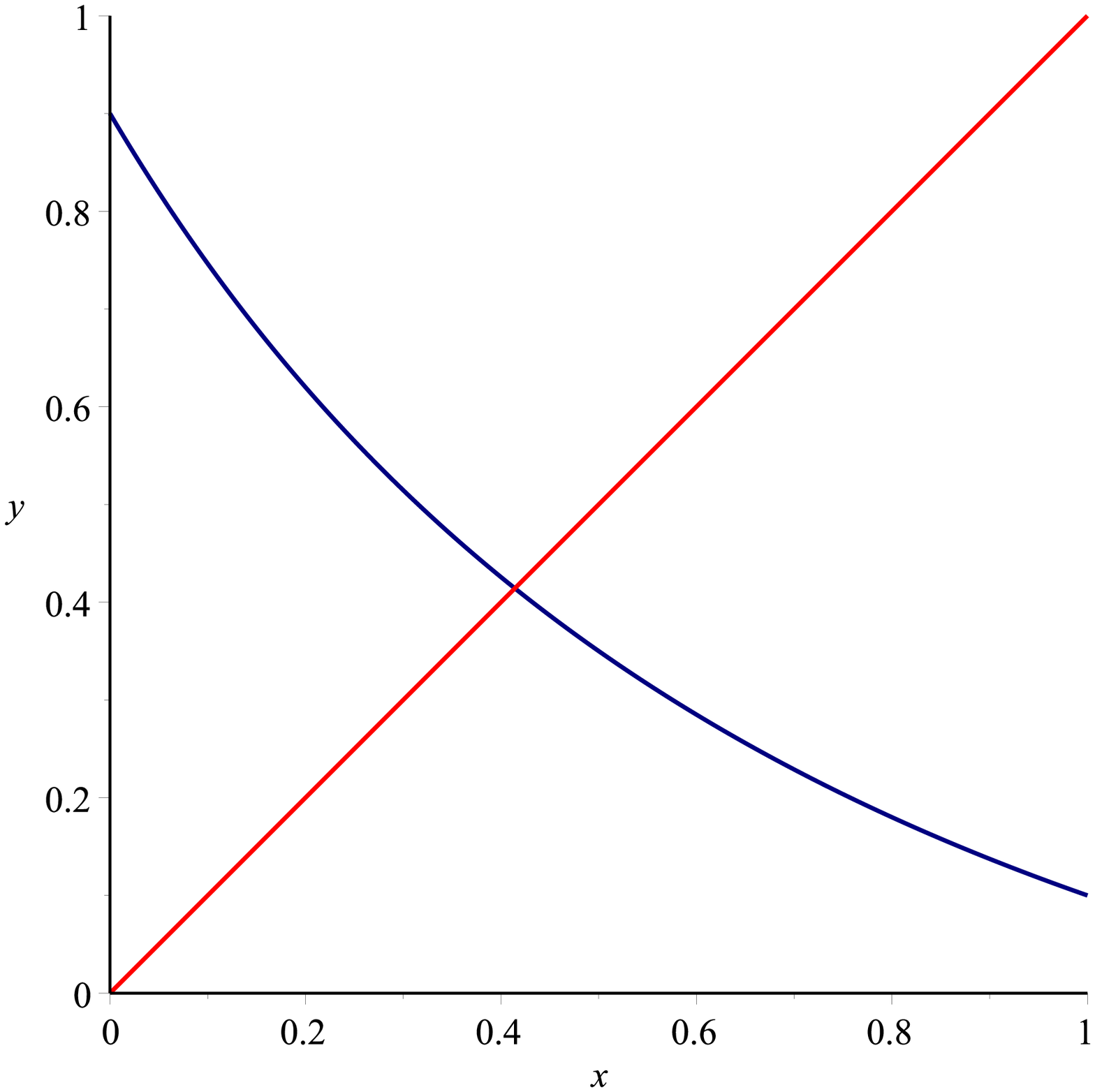}\\
\hfill
\caption{$(\alpha;\beta)\in D$}\label{rD}
\hfill
\end{multicols}
\end{figure}

\begin{figure}[h!]
\begin{multicols}{2}
\hfill
\includegraphics[width=0.4\textwidth]{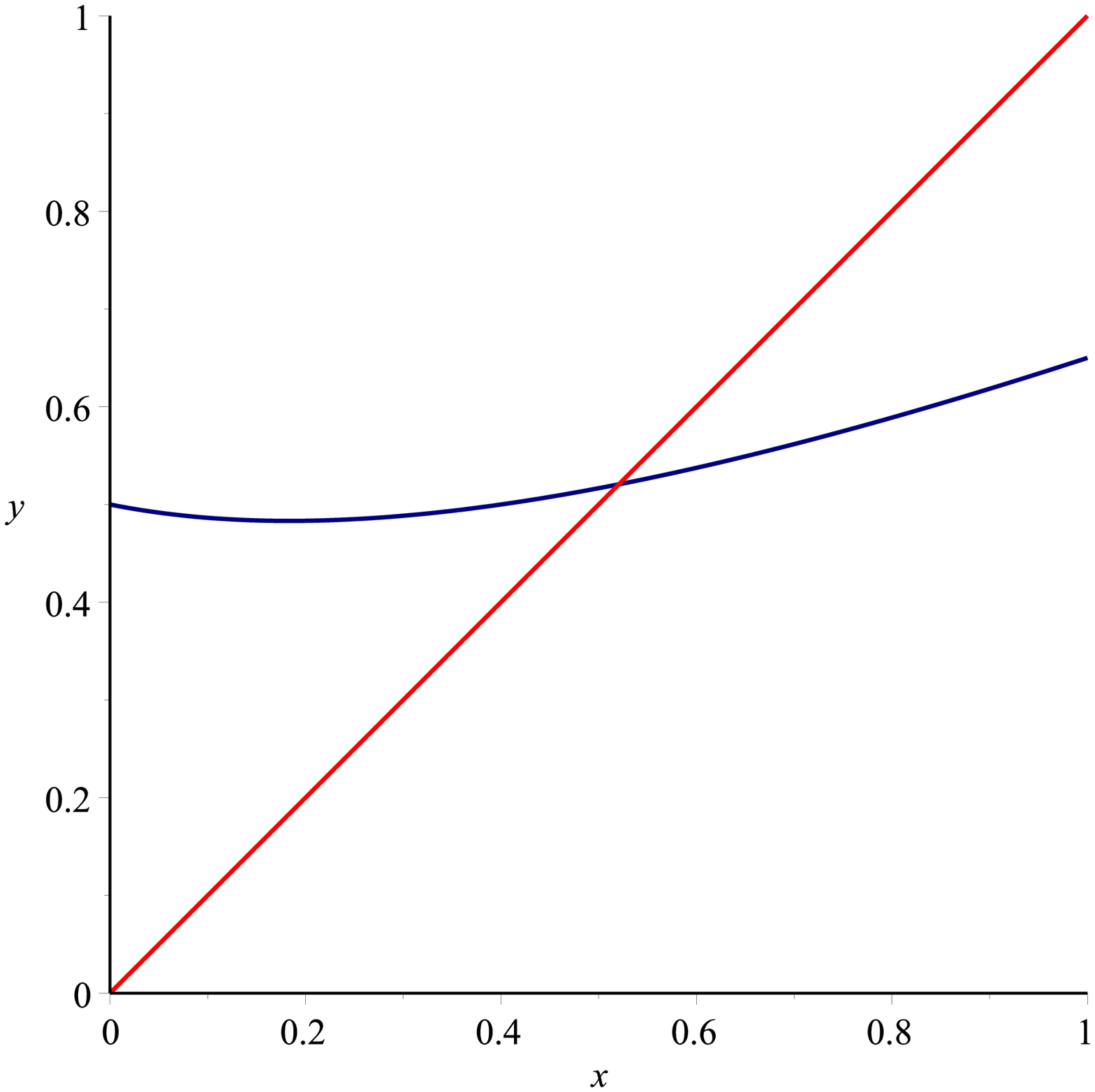}\\
\hfill
\caption{$(\alpha;\beta)\in E^{*}$}\label{rE}
\hfill
\includegraphics[width=0.4\textwidth]{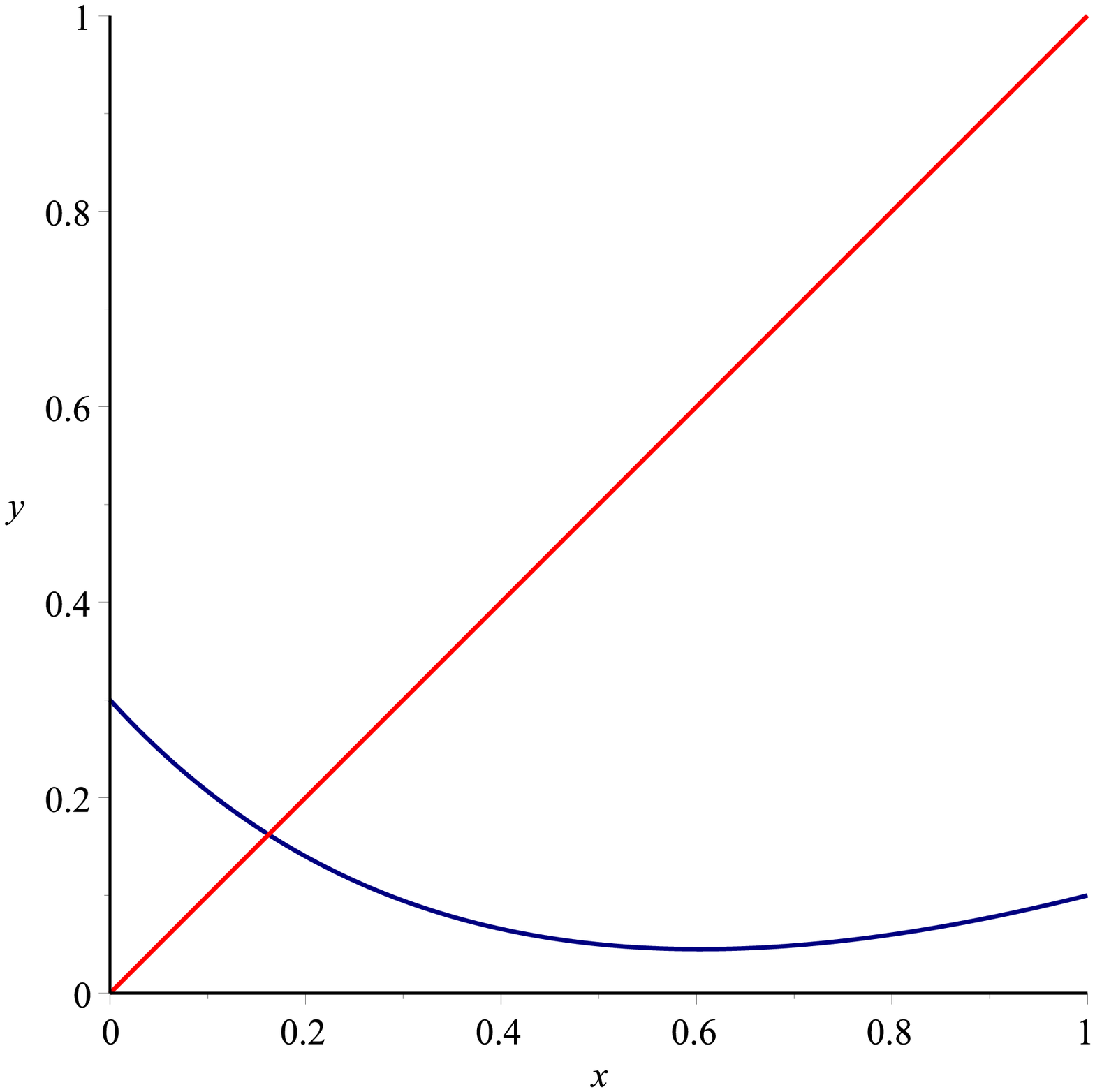}\\
\hfill
\caption{$(\alpha;\beta)\in F^{*}$}\label{rF}
\hfill
\end{multicols}
\end{figure}
As a corollary of proved lemmas we obtain
\begin{thm}\label{to}
Let $z_0=(x_{0}, 1-x_0)\in S$ be an initial point
\begin{itemize}
  \item[1)] If \  $(\alpha;\beta)\in (A\cup B)\setminus\{(2;1)\}$  then
  $$ \lim_{m\to \infty}W_0^{m}(z_{0})=(x^{*}, 1-x^*).$$
  \item[2)] If\ $\alpha=2,\beta=1$ then
  $$\lim_{n\to \infty}W_0^{n}(z_{0})=\left\{\begin{array}{ll}
  (x_{0},1-x_0)\ \ \mbox{for} \ \ n=2k, k=0,1,2,...\\[2mm]
\left(\frac{1-x_{0}}{1+x_{0}}, \, \frac{2x_{0}}{1+x_{0}}\right) \ \ \ \mbox{for} \ \ n=2k-1
\end{array}\right.$$
\end{itemize}
\end{thm}

\section{Biological interpretations}

In biology an population biologist is interested in the long-term behavior of the population of a certain
species or collection of species. Namely, what happens
to an initial population of members. Does the population become arbitrarily large as time goes on? Does the population tend to zero, leading
to extinction of the species?  In this section we briefly give
some answers to these questions related to our model of the mosquito population.

Each point (vector) $z=(x;y)\in \mathbb{R}^{2}_{+}$ can be considered as a state (a measure) of the mosquito
population. In case $x+y\ne 0$ one can consider $z$ as a probability measure (after a normalization if needed).
If, for example, the value of
$x$ is close to zero, biologically this means that the contribution of the larvae class is small
in future of the population.

Let us give some interpretations of our main results:

\begin{itemize}
\item[(a)] (Case Theorem \ref{fixthm}, part a.) The population has a unique equilibrium state;

\item[(b)] (Case Theorem \ref{fixthm}, part b.) Under conditions on parameters the population has exactly two equilibrium states;

\item[(c)] (Case Theorem \ref{fixthm}, part c.) The population has a continuum set of equilibrium states.

\item[(d)] (Case Proposition \ref{pr}) If parameters of the model satisfy the conditions of the proposition then the trajectory of the
population has limit $(0,0)$ (i.e. vanishing of the population) or $(x^*, y^*)$ (i.e. both species will survive).

\item[(e)] (Case Theorem \ref{to})  Under conditions of the part 1) of the theorem  both stages of the population will survive with
probability $x^*$ and $1-x^*$ respectively.
 Under conditions of part 2) of the theorem, for each initial state the population
will have 2-periodic state. Thus any state is 2-periodic.
 \end{itemize}


\begin{thebibliography}{99}
\bibitem{L.Alphey}  L. Alphey, M. Benedict, R. Bellini, G.G. Clark, D.A. Dame, M.W. Service, and S.L. Dobson, \textit{Sterile-insect methods for control of mosquito-borne diseases}:  An analysis, Vector Borne Zoonotic Dis. {\bf 10} (2010), 295--311.

\bibitem{A.Bartlet}  A.C. Bartlettand, R. T.Staten, \textit{Sterile Insect Release Method and Other Genetic Control Strategies},
Radcliffe's IPM World Textbook, 1996.

\bibitem{Becker} N. Becker, \textit{Mosquitoes and Their Control}, Kluwer Academic/Plenum, New York, 2003.

\bibitem{D} R.L. Devaney,  \textit{An Introduction to Chaotic Dynamical System} (Westview Press, 2003).

\bibitem{G}  T. Gerald, \textit{Ordinary Differential Equations and Dynamical Systems}. Providence: American Mathematical Society. 2012.

\bibitem{Mosquito} Mosquito, 2010. Available from: http://www.enchantedlearning.com/subjects/insects/mosquito

\bibitem{J.Li} J. Li, L. Cai, Y. Li, \textit{Stage-structured wild and sterile mosquito population models and their dynamics}.
Journal of Biological Dynamics. {\bf 11}(2) (2017), 79--101.

\bibitem{J} J. Li,  \textit{Malaria model with stage-structured mosquitoes}, Math. Biol.Eng.
\textbf{8} (2011), 753--768.

\bibitem{LP} A. M. Lutambi, M.A.  Penny, T. Smith, N. Chitnis, \textit{Mathematical modelling of mosquito dispersal in a heterogeneous environment}.
 Math. Biosci. \textbf{241}(2) (2013), 198--216.

\bibitem{RAU} U.A. Rozikov, H. Akin, S. Uguz, \textit{Exact solution of a generalized ANNNI model on a Cayley tree}. Math. Phys. Anal. Geom. \textbf{17}(1) (2014), 103--114.

\bibitem{RV}    U.A. Rozikov, M.V. Velasco, \textit{A discrete-time dynamical system and an evolution algebra
of mosquito population}. {\it Jour. Math. Biology.} \textbf{78}(4) (2019), 1225--1244.

\bibitem{RS} U.A. Rozikov, S.K. Shoyimardonov, \textit{On ocean ecosystem discrete time dynamics
generated by $\ell$-Volterra operators}. Inter. Jour. Biomath. \textbf{12}(2) (2019), 1950015, (24 pages).

\end{thebibliography}
\end{document}